\documentclass[11pt]{article}

\usepackage[margin=1in]{geometry}
\usepackage{amsmath,amssymb,amsthm,mathtools}
\usepackage{enumitem}
\usepackage[colorlinks=true,citecolor=blue,linkcolor=blue,urlcolor=blue]{hyperref}
\usepackage{tikz}
\usetikzlibrary{arrows.meta}
\theoremstyle{plain}
\newtheorem{lemma}{Lemma}[section]
\newtheorem{theorem}[lemma]{Theorem}
\newtheorem{proposition}[lemma]{Proposition}

\theoremstyle{definition}

\theoremstyle{remark}
\newtheorem{remark}[lemma]{Remark}

\newcommand{\Z}{\mathbb Z}
\newcommand{\Pp}{\mathbb P}
\newcommand{\Ee}{\mathbb E}

\title{Aperiodic random walks: uniform estimates, local limits and invariance principles}
\author{Jiaming Xu}
\date{}

\begin{document}
\maketitle

\begin{abstract}
We prove uniform endpoint estimates, killed local limits, and
conditioned-bridge invariance principles for triangular arrays of centered,
aperiodic, integer-valued random walks with uniform exponential tails.
The endpoint estimates cover all nonnegative heights, while the local
limits treat positive diffusive and zero endpoints. The increment law may vary with the scaling parameter.
\end{abstract}

\section{Introduction}

We consider the following triangular array of integer-valued random
walks. For \(N=1,2,\ldots\), let \(X_{N,1},X_{N,2},\ldots\) be i.i.d. with
law \(\nu_N\) on \(\Z\), and write
\[
 S_N(k)=\sum_{j=1}^k X_{N,j},\qquad
 \sigma_N^2=\Ee X_{N,1}^2 .
\]
We impose the following hypotheses.
\begin{enumerate}[label=\textup{(A\arabic*)},ref=\textup{(A\arabic*)}]
\item \label{ass:centered}
      \(\Ee X_{N,1}=0\) for every \(N\), and
      \(\sigma_N^2\to\sigma^2\in(0,\infty)\).
\item \label{ass:exponential}
      There is \(\eta_0>0\) such that
      \[
        \sup_N \Ee e^{\eta_0|X_{N,1}|}<\infty .
      \]
\item \label{ass:weak}
      \(\nu_N\Rightarrow\nu\), where \(\nu\) is an integer-valued
      probability law.
\item \label{ass:aperiodic}
      The array is uniformly strongly aperiodic: for every
      \(\delta\in(0,\pi)\), there is \(\rho_\delta<1\) such that,
      for all sufficiently large \(N\),
      \[
        \sup_{\delta\le |t|\le\pi}
        \left|\sum_{z\in\Z}e^{itz}\nu_N(z)\right|
        \le \rho_\delta .
      \]
\end{enumerate}
Assumptions~\ref{ass:exponential} and~\ref{ass:weak} imply that, for every
\(0<\eta<\eta_0\),
\begin{equation}
  \sum_{z\in\Z} e^{\eta|z|}
  |\nu_N(z)-\nu(z)|\longrightarrow0 .
  \label{eq:weighted-tv}
\end{equation}
Indeed, weak convergence gives convergence on every finite subset of
\(\Z\), while the two tails are uniformly negligible by
Assumption~\ref{ass:exponential}. Consequently, all moments converge.
Together with Assumptions~\ref{ass:centered} and~\ref{ass:aperiodic}, this
also shows that the limit law is centered, has variance \(\sigma^2\), and is
aperiodic.

We focus on walks started at a nonnegative height and conditioned to
remain nonnegative. For a fixed increment law, these
objects are well studied: ballot-type estimates control the one-sided
survival probabilities, local limit theorems identify
the killed transition kernel, and invariance principles describe the
conditioned bridge itself; see, for instance,
Addario--Berry--Reed~\cite{ABR} and
Caravenna--Chaumont~\cite{CCbridge}.

When the increment law varies with \(N\), the same conclusions are
expected as soon as the convergence of the laws is sufficiently well
controlled. However, the triangular-array version needed in applications
does not seem to be available in a directly quotable form. The purpose
of this note is to fill this gap under the assumptions above: uniform
exponential tails, convergence of the increment laws, and uniform
aperiodicity. We record uniform endpoint and maximum estimates, killed local
limits, and conditioned-bridge invariance principles. We also use the dual
walk, namely the walk with reversed increments
\(\check\nu_N(z)=\nu_N(-z)\). This is the natural description of bridges
ending at the boundary. We treat it throughout together with the
original walk. Under the assumptions of this note, the two walks satisfy the
same uniform estimates, although the constants in their survival
asymptotics may differ, as recorded in
\eqref{eq:survival-constant-convergence}.

The main motivation comes from the author's work on soft-edge limits of
certain discrete and continuous random matrix ensembles. In
Keating--Xu~\cite{KX}, the forthcoming work~\cite{XuForthcoming}, and
possible extensions of these works, the relevant moment expansions are
governed by conditioned walks whose increment laws have larger support than a
Bernoulli law and may vary with matrix size. This application motivates
the present formulation, but the estimates below are purely random-walk
statements and may be useful elsewhere.

We keep the scope deliberately narrow. Fixed-law results are known in more
general lattice settings, and Caravenna--Chaumont~\cite{CCbridge} treat
walks in stable domains of attraction. The conclusions here should admit
analogous extensions under assumptions similar to
Assumptions~\ref{ass:centered}--\ref{ass:aperiodic}, but we do not pursue
them. By contrast, if
uniform aperiodicity is removed, even the ordinary lattice local limit theorem
can oscillate with \(N\); if the uniform exponential-moment assumption
is weakened to convergence of variances alone, the full-range maximum
estimates below need not hold.
\subsection*{AI statement.} Most of the proofs in this note were
initially generated by ChatGPT 5.6, although the underlying mathematical
ingredients should be standard to experts on random walks. The human author
subsequently checked the arguments line by line, made various revisions, and
takes full responsibility for the correctness of the note.

\section{Uniform estimates}

We begin with an estimate for the unconditioned walks. The exponent
\(-cz^2/(L+|z|+1)\) gives the usual Gaussian tail when \(|z|\le L\) and an
exponential tail when \(|z|\ge L\). The latter is essential because
\(|S_N(L)|\) is unbounded for every \(L\ge1\).

\begin{lemma}[Uniform local and local-maximum bounds]
\label{lem:unconditioned}
There are constants \(C,c>0\), independent of \(N,L,z,R\), such that
for all sufficiently large \(N\), \(L,R\in\Z_{\ge0}\), and \(z\in\Z\),
\begin{align}
 \Pp(S_N(L)=z)
 &\le \frac{C}{\sqrt{L+1}}
 \exp\left\{-c\frac{z^2}{L+|z|+1}\right\},       \label{eq:local-bernstein}\\
 \Pp\left(S_N(L)=z,
       \max_{0\le j\le L}|S_N(j)|\ge R\right)
 &\le \frac{C}{\sqrt{L+1}}
 \exp\left\{-c\frac{R^2}{L+R+1}\right\}.         \label{eq:local-max}
\end{align}
Moreover, for some sufficiently large \(N_0\), the uniform lattice
local central limit theorem holds:
\begin{equation}
 \lim_{L\to\infty}\sup_{N\ge N_0}\sup_{z\in\Z}
 \left|
 \sigma_N\sqrt L\,\Pp(S_N(L)=z)
 -\frac1{\sqrt{2\pi}}
  \exp\left\{-\frac{z^2}{2\sigma_N^2L}\right\}
 \right|=0.
 \label{eq:uniform-local-clt}
\end{equation}
\end{lemma}

\begin{proof}
Write
\[
 \phi_N(u):=\Ee e^{iuX_{N,1}},
 \qquad
 \psi_N(t):=\log\Ee e^{tX_{N,1}} .
\]
We first record the uniform local concentration estimate. By
Assumptions~\ref{ass:centered} and~\ref{ass:exponential},
\(\sigma_N^2\) is bounded away from zero and
\(\Ee |X_{N,1}|^3\) is bounded uniformly in \(N\). Hence, for some
\(\delta,c>0\) and all sufficiently large \(N\),
\begin{equation}\label{eq_1}
 |\phi_N(u)|\le e^{-cu^2},\qquad |u|\le\delta .
\end{equation}
On \(\delta\le |u|\le\pi\), Assumption~\ref{ass:aperiodic} gives
\begin{equation}\label{eq_2}|\phi_N(u)|\le\rho<1.\end{equation}
Fourier inversion therefore gives
\[
 \sup_{z\in\Z}\Pp(S_N(L)=z)
 \le
 \frac1{2\pi}\int_{-\pi}^{\pi}|\phi_N(u)|^L\,du
 \le \frac{C}{\sqrt{L+1}} .
\]
The same Fourier decomposition proves
\eqref{eq:uniform-local-clt}.  Indeed, split the Fourier integral into \(|u|\le\delta\) and
\(\delta\le |u|\le\pi\).  In the first region, the change of variables
\(v=\sigma_N\sqrt L\,u\), the uniform Taylor expansion of \(\phi_N\),
and \eqref{eq_1} give \(L^1(dv)\)-convergence to \(e^{-v^2/2}\).
The second region contributes at most \(C\sqrt L\,\rho^L\) by
\eqref{eq_2}.  Fourier inversion then proves
\eqref{eq:uniform-local-clt}, uniformly in \(z\).

We now add the exponential factor by a tilting argument. For
\(|t|\le t_0\), with \(t_0>0\) small, set
\[
 \nu_{N,t}(x):=e^{tx-\psi_N(t)}\nu_N(x),
 \qquad
 \phi_{N,t}(u):=\sum_{x\in\Z}e^{iux}\nu_{N,t}(x).
\]
The tilted laws are not centered. But it suffices to note that under our assumptions, for $|t|\le t_0$, their exponential moments are bounded in a
common neighborhood of the origin, their variances are bounded above
and away from zero, and their characteristic functions have a uniform
gap away from the origin. In particular,
for sufficiently small \(t_0\), there are constants \(c>0\) and
\(\rho<1\) for which \eqref{eq_1}--\eqref{eq_2} hold uniformly in
\(N\) and \(t\). Thus the preceding Fourier bound
applied to \(\nu_{N,t}\) gives
\[
 \sup_{w\in\Z}\Pp_{N,t}(S_N(L)=w)\le \frac{C}{\sqrt{L+1}},
\]
uniformly in \(N,L,t\). Since the original laws are centered and have a
uniform exponential moment,
\[
 \psi_N(t)\le Ct^2,\qquad |t|\le t_0 .
\]
Moreover
\[
 \Pp(S_N(L)=z)
 =
 e^{-tz+L\psi_N(t)}\Pp_{N,t}(S_N(L)=z).
\]
Choose
\[
 t=\operatorname{sgn}(z)\,\alpha\min\{|z|/L,1\}
\]
when \(L\ge1\), with \(\alpha>0\) small enough. If \(|z|\le L\), then
\(-tz+L\psi_N(t)\le -c z^2/L\); if \(|z|>L\), then
\(-tz+L\psi_N(t)\le -c|z|\). Enlarging \(C\) to cover \(L=0\) gives
\eqref{eq:local-bernstein}.

For the maximum estimate without the endpoint, the process
\(\exp\{tS_N(j)-j\psi_N(t)\}\) is a martingale. Stopping it at the first
entrance into \([R,\infty)\) gives
\[
 \Pp\left(\max_{0\le j\le L}S_N(j)\ge R\right)
 \le \exp\{-tR+L\psi_N(t)\}.
\]
With \(t=\alpha\min\{R/L,1\}\), and then applying the same argument to
\(-S_N\), we obtain
\[
 \Pp\left(\max_{0\le j\le L}|S_N(j)|\ge R\right)
 \le C\exp\left\{-c\frac{R^2}{L+R+1}\right\}.
\]

It remains only to keep the endpoint. Let \(m=\lfloor L/2\rfloor\). If
\(|z|\ge R/2\), the bound follows from \eqref{eq:local-bernstein}. If
\(|z|<R/2\), an exit from \([-R,R]\) occurs either during the first half,
or, after reversing the second half from its endpoint \(z\), the reversed
walk has absolute displacement at least \(R/2\). In either case we apply
the preceding maximal bound to the half that contains this displacement
and the uniform \(C(L+1)^{-1/2}\) local concentration bound to the other
half. This proves \eqref{eq:local-max}.
\end{proof}

\paragraph{Endpoint constants and boundary survival.}

For a fixed \(N\), let
\begin{align*}
 u_{N,n}^{\rightarrow}
 &:=\Pp\bigl(S_N(j)\ge0,\ 1\le j\le n\bigr),\\
 u_{N,n}^{\leftarrow}
 &:=\Pp\bigl(-S_N(j)\ge0,\ 1\le j\le n\bigr).
\end{align*}
These are the survival probabilities for the walk and its dual.

\begin{lemma}[Survival probabilities]
\label{lem:survivalprob}
For some \(c,C>0\), uniformly for all sufficiently large \(N\) and all
\(n\ge0\),
\begin{equation}
 \frac{c}{\sqrt{n+1}}
 \le u_{N,n}^{\rightarrow},u_{N,n}^{\leftarrow}
 \le \frac{C}{\sqrt{n+1}}.
 \label{eq:boundary-survival}
\end{equation}
Moreover the limits
{
\begin{align*}
 \kappa_N^{\rightarrow}
 &:=\lim_{n\to\infty}\sqrt n\,u_{N,n}^{\rightarrow},\\
 \kappa_N^{\leftarrow}
 &:=\lim_{n\to\infty}\sqrt n\,u_{N,n}^{\leftarrow}.
\end{align*}}
exist, are positive, and the convergence is uniform along the array:
\begin{equation}
 \sup_{N\ge N_0}
 \left|\sqrt n\,u_{N,n}^{\rightarrow}-\kappa_N^{\rightarrow}\right|
 +\sup_{N\ge N_0}
 \left|\sqrt n\,u_{N,n}^{\leftarrow}-\kappa_N^{\leftarrow}\right|
 \longrightarrow0.
 \label{eq:uniform-survival-asymptotic}
\end{equation}
Set
\begin{equation*}
 a_N:=\sqrt{\frac\pi2}\frac{\kappa_N^{\rightarrow}}{\sigma_N},
 \qquad
 b_N:=\sqrt{\frac\pi2}\frac{\kappa_N^{\leftarrow}}{\sigma_N}.
\end{equation*}
Then there are \(a,b\in(0,\infty)\) such that
\begin{equation}
 a_N\longrightarrow a,
 \qquad b_N\longrightarrow b.
 \label{eq:survival-constant-convergence}
\end{equation}
\end{lemma}

\begin{proof}
The Sparre
Andersen--Spitzer formula~\cite[Equation~(2.2)]{AurzadaSimon} gives, for
\(0\le s<1\),
\begin{equation*}
 U_N(s):=\sum_{n\ge0}s^n u_{N,n}^{\rightarrow}
 =
 \exp\left\{
   \sum_{n\ge1}\frac{s^n}{n}\Pp(S_N(n)\ge0)
 \right\}.
\end{equation*}
This identity is valid for arbitrary real-valued i.i.d. increments.  Uniform Berry--Esseen, together
with the local bound at zero, yields
\begin{equation*}
 \left|\Pp(S_N(n)\ge0)-\frac12\right|
 \le Cn^{-1/2}.
\end{equation*}
Consequently, with
\[
 R_N(s):=\sum_{n\ge1}\frac{s^n}{n}
 \left(\Pp(S_N(n)\ge0)-\frac12\right),
\]
we have
\begin{equation*}
 U_N(s)=(1-s)^{-1/2}e^{R_N(s)},
 \qquad
 |R_N(1)-R_N(s)|
 \le C\sum_{n\ge1}\frac{1-s^n}{n^{3/2}}
 \le C\sqrt{1-s}.
\end{equation*}
In particular, \(R_N(1)\) is bounded uniformly in \(N\), and
\begin{equation}
 U_N(s)=e^{R_N(1)}(1-s)^{-1/2}
 \bigl(1+O(\sqrt{1-s})\bigr)
 \label{eq:uniform-generating-asymptotic}
\end{equation}
uniformly in $N$ along the array.

We next give a Hardy--Littlewood Tauberian type argument uniform in \(N\).
Let
\(A_N(x)=\sum_{0\le j\le x}u_{N,j}^{\rightarrow}\).  From
\eqref{eq:uniform-generating-asymptotic}, applied at
\(s=e^{-\lambda/n}\), the Laplace transforms of the positive measures
\[
 e^{-R_N(1)}n^{-1/2}
 \sum_{j\ge0}u_{N,j}^{\rightarrow}\,\delta_{j/n}
\]
converge as $n\to\infty$, uniformly in $N$ and for \(\lambda\) in compact subsets of
\((0,\infty)\), to \(\lambda^{-1/2}\).  Consequently,
\begin{equation}
 A_N(x)=\frac{2e^{R_N(1)}}{\sqrt\pi}\sqrt{x}\,(1+o(1)),
 \qquad x\to\infty,
 \label{eq:uniform-partial-survival}
\end{equation}
where the \(o(1)\) is uniform in \(N\). Indeed, if the uniformity failed, pick a sequence
$N_k$ and $x_k\to\infty$ along which the asserted asymptotic does not
hold.  Define the positive measures
\[
  \mu_k
  :=
  e^{-R_{N_k}(1)}x_k^{-1/2}
  \sum_{j\ge0}u_{N_k,j}^{\rightarrow}\,\delta_{j/x_k}.
\]
By \eqref{eq:uniform-generating-asymptotic}, for every $\lambda>0$,
\[
  \widehat\mu_k(\lambda)
  :=
  \int_0^\infty e^{-\lambda y}\,\mu_k(dy)
  =
  e^{-R_{N_k}(1)}x_k^{-1/2}
  U_{N_k}(e^{-\lambda/x_k})
  \longrightarrow \lambda^{-1/2}.
\]
The continuity theorem for Laplace transforms of locally finite
positive measures therefore gives the vague convergence of \(\mu_k\) to
\(dy/\sqrt{\pi y}\), which implies
\[
  e^{-R_{N_k}(1)}x_k^{-1/2}A_{N_k}(x_k)
  =\mu_k([0,1])
  \longrightarrow \frac{2}{\sqrt\pi},
\]
contradicting the choice of $(N_k,x_k)$. 

Finally, for every fixed \(\lambda>1\), monotonicity of
\(u_{N,n}^{\rightarrow}\) gives
\begin{align*}
 u_{N,n}^{\rightarrow}
 &\le
 \frac{A_N(n)-A_N(\lfloor n/\lambda\rfloor)}
      {n-\lfloor n/\lambda\rfloor},\\
 u_{N,n}^{\rightarrow}
 &\ge
 \frac{A_N(\lfloor\lambda n\rfloor)-A_N(n)}
      {\lfloor\lambda n\rfloor-n}.
\end{align*}
Insert \eqref{eq:uniform-partial-survival}, let \(n\to\infty\), and
then let \(\lambda\downarrow1\).  This proves
\eqref{eq:uniform-survival-asymptotic} and the explicit formula
\begin{equation}
 \kappa_N^{\rightarrow}
 =
 \frac1{\sqrt\pi}
 \exp\left\{
   \sum_{n\ge1}\frac1n
   \left(\Pp(S_N(n)\ge0)-\frac12\right)
 \right\}.
 \label{eq:spitzer-constant}
\end{equation}
The two-sided estimate \eqref{eq:boundary-survival} follows from the same
uniform argument (enlarging the constants for bounded \(n\)).  Applying
it to the dual array proves all corresponding left-arrow statements.
The summand in \eqref{eq:spitzer-constant} is \(O(n^{-3/2})\) uniformly
in \(N\), while \eqref{eq:weighted-tv} gives convergence for each fixed
\(n\).  Dominated convergence, together with
\(\sigma_N\to\sigma\), proves
\eqref{eq:survival-constant-convergence}.
\end{proof}

For \(x,y,L\in\Z_{\ge0}\), put
\[
 \tau_{N,x}:=\inf\{j\ge1:x+S_N(j)<0\}.
\]

\begin{lemma}[Height-dependent survival probability]
\label{lem:height-survival}
There is a constant \(C>0\) such that, for all sufficiently large \(N\)
and all \(n,H\in\Z_{\ge0}\),
\begin{equation}
 \Pp(\tau_{N,H}>n)
 \le C\min\left\{1,\frac{H+1}{\sqrt{n+1}}\right\}.
 \label{eq:height-survival}
\end{equation}
The same estimate holds for the dual array.
\end{lemma}

\begin{proof}
This is the triangular-array version of the height-dependent survival
estimate from Addario--Berry--Reed~\cite[Lemma~3]{ABR}.
We give the short proof needed here to keep the constants uniform in
\(N\). The case \(H=0\) follows from \eqref{eq:boundary-survival}, and the
estimate is trivial when \(H+1\ge\sqrt{n+1}\). Thus, assume \(H\ge1\) and
\(H+1<\sqrt{n+1}\). Choose \(r_H=\lfloor K(H+1)^2\rfloor\), where \(K\)
is a sufficiently large fixed constant. By \eqref{eq:boundary-survival},
\(u_{N,r_H}^{\rightarrow}\ge c(H+1)^{-1}\). The uniform central limit
theorem gives \(\Pp(S_N(r_H)\ge H)\ge c\) for a uniform constant \(c>0\).
Because the increments in the first \(r_H\) steps are independent and the
indicators of both events \(\{\tau_{N,0}>r_H\}\) and
\(\{S_N(r_H)\ge H\}\) are increasing functions of the increments, the
Harris inequality, in the form established by
Esary--Proschan--Walkup~\cite{EPW}, gives
\begin{equation}\label{eq:harrisoutput}
 \Pp\bigl(\tau_{N,0}>r_H,\ S_N(r_H)\ge H\bigr)
 \ge \Pp\bigl(\tau_{N,0}>r_H\bigr)
      \Pp\bigl(S_N(r_H)\ge H\bigr)
 \ge \frac{c}{H+1}.
\end{equation}
Conditioning on the height at time \(r_H\) and using monotonicity of
the survival probability in the starting height, gives
\[
 u_{N,r_H+n}^{\rightarrow}
 \ge
 \frac{c}{H+1}\,\Pp(\tau_{N,H}>n).
\]
Since \(r_H\le C(H+1)^2\le Cn\), the upper bound in
\eqref{eq:boundary-survival}, applied to \(r_H+n\), proves
\eqref{eq:height-survival}. The dual estimate is identical.
\end{proof}

Define
\[
 q_{N,L}(x,y):=
 \Pp\bigl(x+S_N(L)=y,\ \tau_{N,x}>L\bigr).
\]
If \(E\) is an additional path event, \(q_{N,L}(x,y;E)\) denotes the
same probability with the additional event \(E\) imposed. Thus,
\(q_{N,L}\) is the transition kernel of the walk killed upon entering the
negative half-line.

\paragraph{The remaining one-sided inputs.}
There are constants \(C,c>0\) such that, for all sufficiently large
\(N\) and all \(n,L,H,u\in\Z_{\ge0}\),
\begin{align}
 q_{N,L}(0,H)
 &\le C\frac{H+1}{(L+1)^{3/2}},                                  \label{eq:ABR-centered}\\
 \Pp\bigl(\tau_{N,0}>n,\ S_N(n)\ge u\bigr)
 &\le \frac{C}{\sqrt{n+1}}
 \exp\left\{-c\frac{u^2}{n+u+1}\right\}.                        \label{eq:meander-moderate}
\end{align}
The same estimates hold for the dual array.

For \eqref{eq:ABR-centered}, we split the walk into three parts. For
\(L\ge4\), set \(m=\lfloor L/4\rfloor\). If a path from \(0\) to
\(H\) remains nonnegative, then its first \(m\) steps stay above zero; after
reversing the final \(m\) increments, that block is a dual walk started at
height \(H\) and constrained to stay nonnegative; and, once these two outer
blocks are fixed, the middle block is forced to have one prescribed
lattice displacement. The three blocks are independent. By
\eqref{eq:boundary-survival}, \eqref{eq:height-survival}, and the uniform
local concentration estimate in Lemma~\ref{lem:unconditioned}, their
respective costs are bounded by
\[
 \frac{C}{\sqrt{L+1}},\qquad
 C\min\left\{1,\frac{H+1}{\sqrt{L+1}}\right\},\qquad
 \frac{C}{\sqrt{L+1}}.
\]
This gives \eqref{eq:ABR-centered}; the finitely many cases \(L<4\) are
absorbed into \(C\).

Finally, we prove \eqref{eq:meander-moderate}. For
\(0\le t\le t_0\), put
\[
 \nu_{N,t}(x)=e^{tx-\psi_N(t)}\nu_N(x),
 \qquad \psi_N(t)=\log\Ee e^{tX_{N,1}}.
\]
Let
\[
 U_{N,t}(s):=\sum_{n\ge0}s^n\Pp_{N,t}(\tau_{N,0}>n).
\]
The uniform tilted Berry--Esseen bound gives, for \(k\ge1\),
\[
 \left|\Pp_{N,t}(S_N(k)\ge0)-\frac12\right|
 \le C\bigl(k^{-1/2}+t\sqrt{k}\bigr)
\]
whenever \(t\sqrt{k}\le1\), and the trivial bound otherwise.  Put
\(a=1-s\), and note that \[\sum_{k\ge 1}s^kk^{-1/2}\le \frac{C}{\sqrt{1-s}}=\frac{C}{\sqrt{a}}\] for small positive $a$.  If \(t\le\sqrt a\), insertion in the
Sparre Andersen--Spitzer factorization gives
\[
 \log U_{N,t}(s)
 \le \frac12\log\frac1a+C,
\]
because \(t\sum_{k\ge1}s^k k^{-1/2}\le C\).  If
\(t>\sqrt a\), rewrite the same identity as
\[
 U_{N,t}(s)=\frac1{1-s}
 \exp\left\{-\sum_{k\ge1}\frac{s^k}{k}
                 \Pp_{N,t}(S_N(k)<0)\right\}.
\]
For \(k\le t^{-2}\), the lower Berry--Esseen estimate gives
\(\Pp_{N,t}(S_N(k)<0)\ge
1/2-C(t\sqrt{k}+k^{-1/2})\).  Hence
\begin{align*}
 &\sum_{k\ge1}\frac{s^k}{k}\Pp_{N,t}(S_N(k)<0)
 \ge \frac12\sum_{k=1}^{t^{-2}}\frac{s^k}{k}-Ct\sum_{k=1}^{t^{-2}}\frac{1}{\sqrt{k}}-C\\
 =&\frac{1}{2}\sum_{k=1}^{t^{-2}}\frac{1}{k}-\frac{1}{2}\sum_{k=1}^{t^{-2}}\frac{1-(1-a)^{k}}{k}-C\ge\log\frac{1}{t}-\frac{1}{2}at^{-2}-C\ge  \log\frac1t-C;
\end{align*}
the last step uses \(a<t^2\).  Thus, in both regimes,
\begin{equation*}
 U_{N,t}(s)
 \le C\bigl((1-s)^{-1/2}+t(1-s)^{-1}\bigr).
\end{equation*}
Since the coefficients are nonnegative, evaluating the generating-function
bound at \(s=e^{-1/(n+1)}\) gives
\[
 \sum_{j=0}^n \Pp_{N,t}(\tau_{N,0}>j)
 \le C\bigl(\sqrt{n+1}+t(n+1)\bigr).
\]
The survival probabilities are decreasing in \(j\), hence
\begin{equation}\label{eq:survivingprobtilted}
 \Pp_{N,t}(\tau_{N,0}>n)
 \le C\bigl((n+1)^{-1/2}+t\bigr).
\end{equation}
By change of measure and \(\psi_N(t)\le Ct^2\),
\[
 \Pp\bigl(\tau_{N,0}>n,S_N(n)\ge u\bigr)
 \le e^{-tu+n\psi_N(t)}\Pp_{N,t}(\tau_{N,0}>n).
\]
If \(u\le n+1\), choose \(t=\alpha u/(n+1)\), with \(\alpha>0\)
small.  The preceding display gives
\[
 \Pp\bigl(\tau_{N,0}>n,S_N(n)\ge u\bigr)
 \le
 \frac{C}{\sqrt{n+1}}
 \left(1+\frac{u}{\sqrt{n+1}}\right)
 e^{-c u^2/(n+1)}.
\]
Absorbing the factor $\frac{u}{\sqrt{n+1}}$ by the exponential term after decreasing \(c\) proves
\eqref{eq:meander-moderate} in this range. If \(u>n+1\), take a fixed
small \(t=\alpha\).  Then change of measure gives \(Ce^{-cu}\); since
\(u>n\), the missing factor \((n+1)^{-1/2}\) is again absorbed by the exponential after decreasing
\(c\).  This proves 
\eqref{eq:meander-moderate}. The dual estimates are identical.

\begin{proposition}[Uniform one-sided endpoint and maximum estimates]
\label{prop:killed-bounds}
There are constants \(C,c>0\) such that, for all sufficiently large
\(N\) and all \(L,H,R\in\Z_{\ge0}\),
\begin{align}
 q_{N,L}(0,H)
 &\le C\frac{H+1}{(L+1)^{3/2}}
 \exp\left\{-c\frac{H^2}{L+H+1}\right\},         \label{eq:forward-endpoint}\\
 q_{N,L}\bigl(0,H;\max_{j\le L}S_N(j)-H\ge R\bigr)
 &\le C\frac{H+1}{(L+1)^{3/2}}
 \exp\left\{
   -c\frac{H^2}{L+H+1}
   -c\frac{R^2}{L+R+1}
 \right\}.                   \label{eq:forward-relative-max}
\end{align}
Consequently,
\begin{align}
 \sum_{H\ge0}
 q_{N,L}\bigl(0,H;\max_{j\le L}S_N(j)\ge R\bigr)
 &\le \frac{C}{\sqrt{L+1}}
 \exp\left\{-c\frac{R^2}{L+R+1}\right\},
 \label{eq:summed-reverse}\\
 q_{N,L}\bigl(0,0;\max_{j\le L}S_N(j)\ge R\bigr)
 &\le \frac{C}{(L+1)^{3/2}}
 \exp\left\{-c\frac{R^2}{L+R+1}\right\}.
 \label{eq:excursion-max}
\end{align}
The same statements hold for the dual array.
\end{proposition}
\begin{remark}
Because of the exponential tail, the prefactor
\((H+1)/(L+1)^{3/2}\) in \eqref{eq:forward-endpoint} and
\eqref{eq:forward-relative-max} may be replaced by the weaker prefactor
\((L+1)^{-1}\).
\end{remark}

\begin{proof}
The case \(L=0\) is immediate after enlarging \(C\), so assume \(L\ge1\)
and write \(n=L+1\).

We first prove \eqref{eq:forward-endpoint}. If \(H\le\sqrt n\), then
\eqref{eq:ABR-centered} already gives the result, since
\(H^2/(L+H+1)\le1\).

Suppose next that \(\sqrt n<H\le L\). Split at
\(m=\lceil L/2\rceil\) and put \(r=L-m\). Dropping the positivity
constraint after time \(m\) and using the Markov property,
\[
 q_{N,L}(0,H)
 \le \sum_{x\ge0}q_{N,m}(0,x)\Pp(S_N(r)=H-x)
 =:I_1+I_2,
\]
where \(I_1\) is the sum over \(x\le H/2\) and \(I_2\) the sum over
\(x>H/2\). For \(I_1\), \eqref{eq:local-bernstein} gives, uniformly in
\(x\le H/2\),
\[
 \Pp(S_N(r)=H-x)
 \le \frac{C}{\sqrt n}
 \exp\left\{-c\frac{H^2}{L+H+1}\right\}.
\]
Summing the killed kernel over \(x\) and using
\eqref{eq:boundary-survival},
\[
 I_1
 \le \frac{C}{n}
 \exp\left\{-c\frac{H^2}{L+H+1}\right\}.
\]
For \(I_2\), use only the uniform local concentration bound on the second
half and \eqref{eq:meander-moderate} on the first. Since \(H\le L\) and
\(m=\lceil L/2\rceil\), we have \(H/2\le m\), and hence
\[
 I_2
 \le \frac{C}{\sqrt n}
 \Pp\bigl(\tau_{N,0}>m,S_N(m)>H/2\bigr)
 \le \frac{C}{n}
 \exp\left\{-c\frac{H^2}{n}\right\}.
\]
Thus, throughout \(\sqrt n<H\le L\),
\[
 q_{N,L}(0,H)
 \le \frac{C}{n}
 \exp\left\{-c\frac{H^2}{L+H+1}\right\}.
\]
Since \(H+1>\sqrt n\), this is stronger than
\eqref{eq:forward-endpoint}.

Finally, if \(H>L\), simply drop the positivity constraint and use
\eqref{eq:local-bernstein}:
\[
 q_{N,L}(0,H)
 \le \frac{C}{\sqrt n}
 \exp\left\{-c\frac{H^2}{L+H+1}\right\}.
\]
Here \(H+1\ge n\), so this is again stronger than
\eqref{eq:forward-endpoint}. This proves the forward endpoint estimate.

We next prove \eqref{eq:forward-relative-max}. For \(0\le B\le A\), set
\[
 q_{N,m}^{[0,A]}(0,B)
 :=\Pp\bigl(S_N(m)=B,\ 0\le S_N(j)\le A,\ 0\le j\le m\bigr)
\]
and put \(d_A(B)=(B+1)\wedge(A-B+1)\).  Then
\begin{equation}
 q_{N,m}^{[0,A]}(0,B)
 \le C\frac{d_A(B)}{(m+1)^{3/2}}
 \exp\left\{-c\frac{B^2}{m+B+1}\right\}.
 \label{eq:strip-endpoint-short}
\end{equation}
The same estimate holds for the dual array.

We include the proof because the factor \(d_A(B)\), especially in the case
\(B=A\), is exactly what makes the maximum decomposition summable.
For \(0\le t\le t_0\), recall from \eqref{eq:survivingprobtilted}
that
\begin{equation}
 \Pp_{N,t}(\tau_{N,0}>m)
 \le C\bigl((m+1)^{-1/2}+t\bigr).
 \label{eq:tilted-boundary-survival-short}
\end{equation}
The height-dependent version is
\begin{equation}
 \Pp_{N,t}(\tau_{N,h}>m)
 \le C(h+1)\bigl((m+1)^{-1/2}+t\bigr).
 \label{eq:tilted-height-survival-short}
\end{equation}
Indeed, when \(h+1\ge\sqrt{m+1}\) this is trivial.  Otherwise choose
\(r_h=\lfloor K(h+1)^2\rfloor\).  The positively tilted law stochastically
dominates the centered law, so by \eqref{eq:harrisoutput},
\[
 \Pp_{N,t}(\tau_{N,0}>r_h,\ S_N(r_h)\ge h)
 \ge \frac{c}{h+1}.
\]
The Markov property and \eqref{eq:tilted-boundary-survival-short}, applied
at time \(r_h+m\), imply
\eqref{eq:tilted-height-survival-short}.

For \(0\le B\le A\), define the analogous kernel under the
tilted law by
\[
 q_{N,m}^{[0,A],t}(0,B)
 :=
 \Pp_{N,t}\bigl(
   S_N(m)=B,\ 
   0\le S_N(j)\le A,\ 0\le j\le m
 \bigr)
\]
To estimate this kernel,
decompose a path of length \(m\) into three consecutive blocks of
comparable lengths, as illustrated in
Figure~\ref{fig:three-block-strip}. \begin{figure}[ht]
\centering
\resizebox{\textwidth}{!}{%
\begin{tikzpicture}[
    x=1cm,
    y=0.72cm,
    original/.style={very thick, blue!70!black},
    transformed/.style={very thick, red!75!black},
    boundary/.style={thick, black},
    cut/.style={densely dashed, gray!75},
    point/.style={circle, fill, inner sep=1.7pt},
    >=Latex
]

\def\A{6}
\def\B{2}


\draw[gray!55] (0,0) -- (8.4,0);
\draw[boundary] (0,\A) -- (8.4,\A)
    node[right] {$A$};

\draw[original]
plot[smooth, tension=0.65] coordinates {
    (0,0)
    (0.6,1.0)
    (1.1,0.6)
    (1.7,2.0)
    (2.3,2.7)
    (2.8,2.2)
    (3.4,3.6)
    (4.0,3.0)
    (4.6,4.4)
    (5.2,3.5)
    (5.6,4.0)
    (6.1,4.8)
    (6.7,3.7)
    (7.2,5.0)
    (7.7,3.6)
    (8.05,2.8)
    (8.4,\B)
};

\draw[cut] (2.8,0) -- (2.8,\A);
\draw[cut] (5.6,0) -- (5.6,\A);

\node[point] at (0,0) {};
\node[below left] at (0,0) {$0$};

\node[point] at (8.4,\B) {};
\draw[densely dotted, gray!75]
    (8.4,\B) -- (8.85,\B)
    node[right, black] {$B$};

\node[below] at (2.8,0) {$m_1$};
\node[below] at (5.6,0) {$m_2$};
\node[below] at (8.4,0) {$m$};

\node[above] at (1.4,\A) {first block};
\node[above] at (4.2,\A) {second block};
\node[above] at (7.0,\A) {third block};

\draw[->, thick]
    (8.9,3.3) -- (11.0,3.3)
    node[midway, above, align=left, font=\small]
    {reverse time\\[-1mm] and use \(A-S_N(\cdot)\)};


\begin{scope}[xshift=11.4cm]

\node[above] at (1.4,\A)
    {transformed third block};

\draw[gray!55] (0,0) -- (2.8,0);
\draw[boundary] (0,\A) -- (2.8,\A)
    node[right] {$A$};

\draw[transformed]
plot[smooth, tension=0.65] coordinates {
    (0,4.0)
    (0.35,3.2)
    (0.70,2.4)
    (1.20,1.0)
    (1.70,2.3)
    (2.30,1.2)
    (2.80,2.0)
};


\node[point] at (0,4.0) {};
\node[above] at (0,4.2) {$A-B$};

\node[point] at (2.8,2.0) {};
\node[right] at (2.8,2.0)
    {$A-S_N(m_2)$};

\node[below] at (0,0) {$0$};
\node[below] at (2.8,0) {$m-m_2$};

\end{scope}

\end{tikzpicture}}
\caption{Three-block decomposition of a path with height in
\([0,A]\), and the time-reversed third block measured downward from the
upper boundary \(A\).}
\label{fig:three-block-strip}
\end{figure}
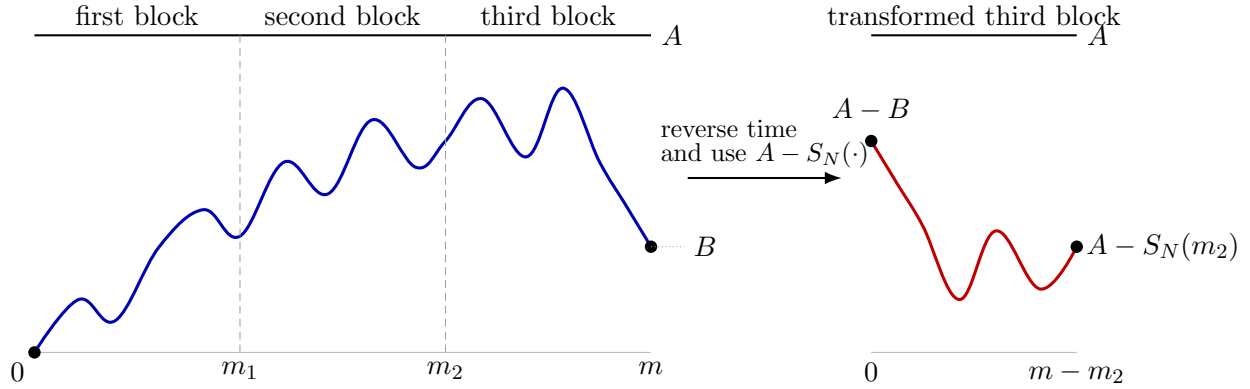 The first block
must survive above the lower boundary.  Reversing the last block from
the endpoint \(B\) gives a dual \((-t)\)-tilted walk started at \(B\),
which is stochastically dominated by the centered dual walk.  Applying
\eqref{eq:tilted-boundary-survival-short} to the first block,
Lemma~\ref{lem:height-survival} to the last block, and the uniform tilted
local concentration bound to the middle block gives
\[
 q_{N,m}^{[0,A],t}(0,B)
 \le C\frac{B+1}{(m+1)^{3/2}}
       \bigl(1+t\sqrt{m+1}\bigr).
\]
Alternatively, reverse the last block in time and measure its height
downward from the upper boundary \(A\).  The
result is a \(t\)-tilted walk started at height \(A-B\).  Using
\eqref{eq:tilted-height-survival-short} on that block gives
\[
 q_{N,m}^{[0,A],t}(0,B)
 \le C\frac{A-B+1}{(m+1)^{3/2}}
       \bigl(1+t\sqrt{m+1}\bigr)^2.
\]
Taking the better bound and changing measure at the fixed endpoint,
\[
 q_{N,m}^{[0,A]}(0,B)
 =e^{-tB+m\psi_N(t)}q_{N,m}^{[0,A],t}(0,B).
\]
With \(t=\alpha B/(m+B+1)\) for some small enough $\alpha>0$, the exponential factor is at most
\(e^{-cB^2/(m+B+1)}\); the factor
\(t\sqrt{m+1}\) is absorbed by the exponential after decreasing \(c\).  This proves
\eqref{eq:strip-endpoint-short}.

We now turn to the maximum.  If \(R\le\sqrt{L+1}\), then
\eqref{eq:forward-endpoint} immediately implies the desired estimate,
since \(R^2/(L+R+1)\le1\).  Assume henceforth that
\(R>\sqrt{L+1}\).  On the event in
\eqref{eq:forward-relative-max}, let
\[
 A:=\max_{0\le j\le L}S_N(j),
 \qquad K:=\min\{j:S_N(j)=A\}.
\]
Then \(A\ge H+R\).  Before \(K\) the path goes from \(0\) to \(A\)
inside \([0,A]\); reflecting the part after \(K\) about \(A\), without
reversing time, turns it into a dual path from \(0\) to \(A-H\) inside
the same strip.  Dropping only the requirement that \(K\) be the first time of
the maximum gives
{
\begin{equation*}
 \begin{split}
 &q_{N,L}\bigl(0,H;\max_{j\le L}S_N(j)-H\ge R\bigr)\\
 &\quad\le
 \sum_{A\ge H+R}\sum_{k=0}^{L}
 q_{N,k}^{[0,A]}(0,A)
 \check q_{N,L-k}^{[0,A]}(0,A-H).
 \end{split}
\end{equation*}}
Write \(B=A-H\) and
\(d=(B+1)\wedge(H+1)\).  Applying
\eqref{eq:strip-endpoint-short} to the two factors yields
\begin{equation}
 C\sum_{B\ge R}d\sum_{k=0}^{L}
 \frac{
 e^{-c(H+B)^2/(k+H+B+1)}
 }
 {(k+1)^{3/2}}\frac{e^{-cB^2/(L-k+B+1)}}{(L-k+1)^{3/2}}.
 \label{eq:max-double-sum-short}
\end{equation}

It remains to estimate \eqref{eq:max-double-sum-short}. For \(x\ge0\), we use the elementary bound
\begin{equation}
 \sum_{m\ge0}\frac1{(m+1)^{3/2}}
 \exp\left\{-c\frac{x^2}{m+x+1}\right\}
 \le \frac{C}{x+1},
 \label{eq:first-passage-sum-short}
\end{equation}
which follows by splitting at \(m=x\) and \(m=x^2\) and bounding the
three resulting sums separately.
Put \(f_L(x)=x^2/(L+x+1)\).  On the first half of the sum,
\(0\le k\le\lfloor L/2\rfloor\), bound the second fraction in
\eqref{eq:max-double-sum-short} by its maximum and sum the first fraction
using \eqref{eq:first-passage-sum-short}.  On the second half, reverse the
roles of the two fractions.  This gives
\begin{align}
 &\sum_{k=0}^{L}
 \frac{e^{-c(H+B)^2/(k+H+B+1)}
       e^{-cB^2/(L-k+B+1)}}
      {(k+1)^{3/2}(L-k+1)^{3/2}}\notag\\
 &\quad\le
 \frac{C}{(L+1)^{3/2}}
 e^{-c'f_L(H+B)-c'f_L(B)}
 \left(\frac1{H+B+1}+\frac1{B+1}\right).
 \label{eq:k-convolution-short}
\end{align}
By definition we have \(d\le H+1\). The parenthesis is
at most \(2/(B+1)\).  Moreover \(f_L(x)\) is increasing,
so 
\(f_L(H+B)\ge f_L(H)\).  Finally,
\begin{equation}
 \sum_{B\ge R}\frac1{B+1}e^{-c f_L(B)}
 \le C e^{-c'f_L(R)},
 \qquad R>\sqrt{L+1},
 \label{eq:large-B-sum-short}
\end{equation}
as follows by splitting at \(B=L+1\) and comparing with the Gaussian
integral below that point and the exponential integral above it.
Combining \eqref{eq:k-convolution-short} and \eqref{eq:large-B-sum-short} proves
\eqref{eq:forward-relative-max}.

To obtain \eqref{eq:summed-reverse}, split the height sum into
\(H\ge R/2\) and \(H<R/2\). In the first range, use
\eqref{eq:meander-moderate}; in the second, use
\eqref{eq:forward-relative-max} with \(R-H\ge R/2\).
Finally, split a \(0\)-to-\(0\) path at
\(m=\lfloor L/2\rfloor\), according to the half on which its maximum is
attained.  For the first-half contribution, the Markov property gives
\[
 \sum_{H\ge0}q_{N,m}(0,H;\max_{j\le m}S_N(j)\ge R)
 q_{N,L-m}(H,0).
\]
By time reversal, the second factor is a dual zero-to-\(H\) kernel and
is bounded by \eqref{eq:forward-endpoint}.  Split again at \(H=R/2\):
use endpoint decay for \(H\ge R/2\), and
\eqref{eq:forward-relative-max} for \(H<R/2\).  The resulting sum is
\[
 \frac{C}{(L+1)^3}e^{-cR^2/(L+R+1)}
 \sum_{H\ge0}(H+1)^2e^{-cH^2/(L+H+1)}
 \le \frac{C}{(L+1)^{3/2}}e^{-c'R^2/(L+R+1)}.
\]
The other half is handled identically for the dual array.  This proves
\eqref{eq:excursion-max}.
\end{proof}

\section{Killed local limits}
Recall that the killed Brownian kernels 
{
\begin{align*}
 \mathbf F(t;y,x)
 &:=\frac1{\sqrt{2\pi t}}
 \left(
 e^{-(y-x)^2/(2t)}-e^{-(y+x)^2/(2t)}
 \right),\\
 \mathbf F_0(t;y)
 &:=\sqrt{\frac2\pi}\frac{y}{t^{3/2}}
 e^{-y^2/(2t)},\\
 \mathbf F_{0,0}(t)
 &:=\sqrt{\frac2\pi}\,t^{-3/2}.
\end{align*}}
Here \(\mathbf F(t;y,x)\) is the transition density from \(x\) to
\(y\) at time \(t\) of Brownian motion killed upon hitting zero.  The
functions \(\mathbf F_0(t;y)\) and \(\mathbf F_{0,0}(t)\) are the
corresponding boundary scaling limits as one or both endpoints tend to
zero.

We record two uniform inputs that will be used in the boundary cases.
First, the same three-block argument as above, now with arbitrary
endpoints, gives
\begin{equation}
 q_{N,L}(x,y)
 \le \frac{C}{\sqrt{L+1}}
 \min\left\{1,\frac{x+1}{\sqrt{L+1}}\right\}
 \min\left\{1,\frac{y+1}{\sqrt{L+1}}\right\}.
 \label{eq:two-boundary-kernel}
\end{equation}
Indeed, use height-dependent survival on the first quarter, the dual
height-dependent survival on the last quarter, and local concentration
on the middle half. 

Second, let \(m_N\to\infty\), and let \(\Pp_{N,m_N}^{+}\) be the law of
the first \(m_N\) steps conditional on \(\tau_{N,0}>m_N\).  
\begin{lemma}[Triangular-array meander convergence]\label{lem:meanderconvergence}

\textup{(a).}
\begin{equation}
 \left(\frac{S_N(\lfloor m_Ns\rfloor)}
 {\sigma_N\sqrt{m_N}}\right)_{0\le s\le1}
 \Longrightarrow (\mathfrak m(s))_{0\le s\le1},
 \label{eq:triangular-meander}
\end{equation}
where \(\mathfrak m\) is the standard Brownian meander. 
In particular,
\(\mathfrak m(1)\) has density \(z e^{-z^2/2}\mathbf1_{z>0}\).

\textup{(b).}
The endpoint \(S_N(m_N)/(\sigma_N\sqrt{m_N})\) is uniformly tight.
\end{lemma}

\begin{proof}
Fix any sequence \(m_N\to\infty\), and define
\[
 \mathcal S_{N,m_N}(t):=
 \frac{S_N(\lfloor m_Nt\rfloor)}{\sigma_N\sqrt{m_N}},
 \qquad t\ge0.
\]
The centering, uniform exponential-moment bound and
\(\sigma_N\to\sigma>0\) imply the Lindeberg condition. Hence the
triangular-array Donsker theorem gives
\(\mathcal S_{N,m_N}\Rightarrow W\) in \(D([0,\infty))\), where
\(W\) is standard Brownian motion.

We now apply the meander invariance principle of
Chaumont--Doney~\cite[Theorem~6]{CDmeander}, which allows the increment
law to vary along the sequence. The centered, nondegenerate walks do
not drift to \(-\infty\), and zero is regular for both open half-lines
for \(W\), as required there. Their discrete conditioning is
\(S_N(j)\ge0\) for \(1\le j\le m_N\), exactly the event
\(\tau_{N,0}>m_N\). The theorem therefore gives
\[
 \left(
   (\mathcal S_{N,m_N}(s))_{0\le s\le1}
   \,\middle|\,\tau_{N,0}>m_N
 \right)
 \Longrightarrow \mathfrak m
 \qquad\text{in }D([0,1]).
\]
This proves \eqref{eq:triangular-meander} along every sequence
\(m_N\to\infty\). Part (b) follows from (a), since evaluation at the terminal time
is continuous on \(D([0,1])\) with the \(J_1\) topology.
\end{proof}

\begin{theorem}[Triangular-array killed local limits]
\label{thm:killed-local-limits}
Let \(r_N\to\infty\), let \(t>0\), and suppose
\[
 L_N=\lfloor tr_N\rfloor,\qquad
 \frac{A_N}{\sigma_N\sqrt{r_N}}\longrightarrow h_0,\qquad
 \frac{B_N}{\sigma_N\sqrt{r_N}}\longrightarrow h_1 .
\]
Whenever \(h_0=0\), we assume \(A_N=0\), and whenever \(h_1=0\),
we assume \(B_N=0\).  
Whenever the endpoints are positive in the limit,
\begin{equation}
 \sqrt{r_N}\,q_{N,L_N}(A_N,B_N)
 \longrightarrow
 \frac1\sigma\mathbf F(t;h_1,h_0),
 \qquad h_0,h_1>0.                                  \label{eq:pp-limit}
\end{equation}
At a zero endpoint, let the constants \(a,b\) be the limits defined
in \eqref{eq:survival-constant-convergence}. Then
\begin{align}
 r_N q_{N,L_N}(0,B_N)
 &\longrightarrow a\,\mathbf F_0(t;h_1),
 && h_1>0,                                            \label{eq:0p-limit}\\
 r_N q_{N,L_N}(A_N,0)
 &\longrightarrow b\,\mathbf F_0(t;h_0),
 && h_0>0,                                            \label{eq:p0-limit}\\
 r_N^{3/2} q_{N,L_N}(0,0)
 &\longrightarrow \sigma ab\,\mathbf F_{0,0}(t).
                                                           \label{eq:00-limit}
\end{align}
The convergences are uniform when \(t\) ranges over a compact subset of
\((0,\infty)\) and the positive limiting endpoints range over compact
subsets of \((0,\infty)\).
\end{theorem}

\begin{proof}
We first prove \eqref{eq:pp-limit}.  The ordinary lattice local CLT is
uniform along the array by \eqref{eq:uniform-local-clt}.  The
triangular-array Donsker theorem and
the local CLT imply convergence of the unconditioned lattice bridges,
uniformly when their two scaled endpoints range in compact subsets of
\((0,\infty)\).  The probability that such a bridge stays nonnegative
then converges to the corresponding Brownian bridge probability.  The
boundary of the positivity event has Brownian probability zero, and
\eqref{eq:two-boundary-kernel} supplies the required uniform domination.
Multiplying the bridge probability by the local endpoint mass proves
\eqref{eq:pp-limit}.  

For the proof of \eqref{eq:0p-limit}, fix \(\varepsilon\in(0,t)\), put
\(m_N=\lfloor\varepsilon r_N\rfloor\), and write
\(\Ee_{N,m_N}^{+}\) for expectation under the length-\(m_N\) conditioned
meander.  The Markov
property gives the exact identity
\begin{equation*}
 q_{N,L_N}(0,B_N)
 =u_{N,m_N}^{\rightarrow}
 \Ee_{N,m_N}^{+}\!\left[
 q_{N,L_N-m_N}(S_N(m_N),B_N)
 \right].
\end{equation*}
By \eqref{eq:uniform-survival-asymptotic},
\eqref{eq:triangular-meander}, and \eqref{eq:pp-limit}, the right-hand
side, after multiplication by \(r_N\), converges to
\begin{equation}
 \frac{\kappa^{\rightarrow}}{\sigma\sqrt\varepsilon}
 \int_0^\infty z e^{-z^2/2}
 \mathbf F(t-\varepsilon;h_1,\sqrt\varepsilon z)\,dz,
 \label{eq:one-boundary-Brownian-convolution}
\end{equation}
where \(\kappa^{\rightarrow}=\lim_N\kappa_N^{\rightarrow}\).
To justify taking the limit inside the expectation, first restrict the
variable
\[
 z=\frac{S_N(m_N)}{\sigma_N\sqrt{m_N}}
\]
to \([\delta,M]\), where
\eqref{eq:pp-limit} is uniform.  Then let \(M\to\infty\), using the
endpoint tail bound in Lemma \ref{lem:meanderconvergence} (b), and let \(\delta\downarrow0\) using
the factor \((x+1)/\sqrt{r_N}\) in
\eqref{eq:two-boundary-kernel}.  

For every \(0<\varepsilon<t\), the Brownian kernel convolution gives
\begin{equation*}
 \frac1{\sqrt\varepsilon}
 \int_0^\infty z e^{-z^2/2}
 \mathbf F(t-\varepsilon;h,\sqrt\varepsilon z)\,dz
 =\frac{h}{t^{3/2}}e^{-h^2/(2t)}.
\end{equation*}
Since
\(a=\sqrt{\pi/2}\,\kappa^{\rightarrow}/\sigma\), the expression in
\eqref{eq:one-boundary-Brownian-convolution} equals
\(a\mathbf F_0(t;h_1)\).  This proves \eqref{eq:0p-limit}.  Applying the
same argument after time reversal to the dual array proves
\eqref{eq:p0-limit}.

It remains to identify the \(0\)-to-\(0\) constant.  Fix
\(s\in(0,t)\) and split at \(\lfloor sr_N\rfloor\):
\[
 q_{N,L_N}(0,0)
 =\sum_{H\ge0}
 q_{N,\lfloor sr_N\rfloor}(0,H)
 q_{N,L_N-\lfloor sr_N\rfloor}(H,0).
\]
Similar as above, one can restrict $H/\sigma_N\sqrt{r_N}$ to $[\delta,M]$. For terms with
\(H=\lfloor\sigma_N\sqrt{r_N}h\rfloor\), use
\eqref{eq:0p-limit} and \eqref{eq:p0-limit}; the lattice Riemann-sum mesh
contributes the factor \(\sigma\).  Therefore
\[
 \lim_{N\to\infty}r_N^{3/2}q_{N,L_N}(0,0)
 =\sigma ab\int_0^\infty
 \mathbf F_0(s;h)\mathbf F_0(t-s;h)\,dh
 =\sigma ab\,\mathbf F_{0,0}(t),
\]
which is \eqref{eq:00-limit}.

\end{proof}

\section{Invariance principle}

Let \(\mathbf Q_{N;A,B,L}^{+}\) denote the law of
\((A+S_N(j))_{0\le j\le L}\) conditional on
\(\tau_{N,A}>L\) and \(A+S_N(L)=B\).

\begin{theorem}[Triangular-array conditioned bridge invariance principle]
\label{thm:bridge-fclt}
Under the assumptions of Theorem~\ref{thm:killed-local-limits}, define
the right-continuous step process
\[
 \widehat S_N(u)
 :=
 \frac{A_N+S_N(\lfloor r_Nu\rfloor)}
      {\sigma_N\sqrt{r_N}},
 \qquad 0\le u\le t .
\]
Under \(\mathbf Q_{N;A_N,B_N,L_N}^{+}\), the process
\(\widehat S_N\) converges in \(D([0,t])\), equipped with the \(J_1\)
topology, to the Brownian bridge of length \(t\) from \(h_0\) to \(h_1\),
conditioned to stay nonnegative. When exactly one endpoint is zero, this is
the corresponding
\(\mathrm{BES}(3)\) bridge, interpreted by time reversal if the terminal
endpoint is zero; when both endpoints are zero it is the normalized
Brownian excursion of length \(t\).
\end{theorem}

\begin{proof}
Write \(\mathbf Q_{x,y}^{+,u}\) for the Brownian bridge of length
\(u\) from \(x\) to \(y\), conditioned to remain nonnegative, and let
\(\mathsf M^u\) denote the Brownian meander of length \(u\).
All discrete blocks below are rescaled in time by \(r_N\) and in space
by \(\sigma_N\sqrt{r_N}\), where integer-part discrepancies are negligible.

\smallskip
\noindent
\emph{Two positive endpoints.}
Suppose \(h_0,h_1>0\).  The triangular-array Donsker theorem and the
uniform lattice local CLT \eqref{eq:uniform-local-clt} give
convergence of the unconditioned lattice
bridge from \(A_N\) to \(B_N\) to the Brownian bridge from \(h_0\) to
\(h_1\).  
For positive endpoints, the event that the bridge remains nonnegative
is a continuity event for the limiting Brownian bridge, and its
probability is
\[
 1-\exp\left\{-\frac{2h_0h_1}{t}\right\}>0.
\]
Conditioning therefore gives
\[
 \mathbf Q_{N;A_N,B_N,L_N}^{+}
 \Longrightarrow
 \mathbf Q_{h_0,h_1}^{+,t}.
\]

\smallskip
\noindent
\emph{One zero endpoint.}
We identify a path with the pair of its restrictions to the two
half-intervals.  By the standard argument on the product space, it suffices to show convergence of
the expectations of product test functions, together with the joint tightness of the pair.  

Assume \(A_N=0\) and \(h_1>0\).  Set
\[
 s:=\frac t2,\qquad
 m_N:=\lfloor sr_N\rfloor,\qquad
 n_N:=L_N-m_N,
\]
and let \(\mathsf M_{N,m_N}:=\Pp_{N,m_N}^{+}\) be the length-\(m_N\)
meander law.  For \(x\ge0\), put
\[
 D_N^{0,h_1}(x)
 :=
 \frac{
   u_{N,m_N}^{\rightarrow}q_{N,n_N}(x,B_N)
 }{
   q_{N,L_N}(0,B_N)
 }.
\]
If \(F\) and \(G\) are bounded continuous functionals of the rescaled
first and second blocks, respectively, the Markov property gives
\[
\begin{aligned}
 &\Ee_{\mathbf Q_{N;0,B_N,L_N}^{+}}[FG] \\
 &\qquad =
 \Ee_{\mathsf M_{N,m_N}}
 \left[
   F\,D_N^{0,h_1}\bigl(S_N(m_N)\bigr)
   \Ee_{\mathbf Q_{N;S_N(m_N),B_N,n_N}^{+}}[G]
 \right].
\end{aligned}
\]
Thus the first block is an endpoint-dependent change of measure of the
meander.

The boundary-survival bounds, \eqref{eq:two-boundary-kernel}, and the
lower limit in \eqref{eq:0p-limit} give
\[
 D_N^{0,h_1}(x)
 \le
 C\min\left\{1,\frac{x+1}{\sqrt{r_N}}\right\}.
\]
Consequently, midpoint heights below
\(\delta\sigma_N\sqrt{r_N}\) have uniformly negligible weighted
meander mass as \(\delta\downarrow0\).  Heights above
\(M\sigma_N\sqrt{r_N}\) are negligible as \(M\to\infty\), because the
density is uniformly bounded and the meander endpoints are uniformly
tight. We may therefore restrict the midpoint to a compact positive
window.  On such a window, convergence along every convergent endpoint
sequence, together with continuity of the limiting kernels, makes the
convergence below uniform by compactness.

If
\[
 \frac{x_N}{\sigma_N\sqrt{r_N}}\longrightarrow z>0,
\]
then the survival and killed local limits give
\[
 D_N^{0,h_1}(x_N)
 \longrightarrow
 \sqrt{\frac{2}{\pi s}}\,
 \frac{\mathbf F(t-s;h_1,z)}
      {\mathbf F_0(t;h_1)},
\]
while the positive-to-positive case gives
\[
 \mathbf Q_{N;x_N,B_N,n_N}^{+}
 \Longrightarrow
 \mathbf Q_{z,h_1}^{+,t-s}.
\]
Integrating over the midpoint height therefore gives convergence of
\(\Ee_{\mathbf Q_{N;0,B_N,L_N}^{+}}[FG]\); we now identify its limit.
The endpoint of \(\mathsf M^s\) has density
\[
 \rho_s(z)=\frac{z}{s}e^{-z^2/(2s)}\mathbf1_{\{z>0\}},
\]
and
\[
 \rho_s(z)
 \sqrt{\frac{2}{\pi s}}\,
 \frac{\mathbf F(t-s;h_1,z)}
      {\mathbf F_0(t;h_1)}
 =
 \frac{
   \mathbf F_0(s;z)\mathbf F(t-s;h_1,z)
 }{
   \mathbf F_0(t;h_1)
 }.
\]
This is precisely the midpoint decomposition of
\(\mathbf Q_{0,h_1}^{+,t}\).  The same compact cutoff gives joint
tightness of the two blocks, and concatenation is continuous at their
continuous limiting paths.  Hence
\[
 \mathbf Q_{N;0,B_N,L_N}^{+}
 \Longrightarrow
 \mathbf Q_{0,h_1}^{+,t}.
\]
Applying the same result to the dual array and reversing time gives the
case \(h_0>0\), \(B_N=0\).

\smallskip
\noindent
\emph{Two zero endpoints.}
Assume \(A_N=B_N=0\), and retain \(s,m_N,n_N\) from above.  Define
\[
 D_N^{0,0}(x)
 :=
 \frac{
   u_{N,m_N}^{\rightarrow}q_{N,n_N}(x,0)
 }{
   q_{N,L_N}(0,0)
 }.
\]
After reversing the second block, its conditional law given the
midpoint \(x\) is the dual bridge
\(\widetilde{\mathbf Q}_{N;0,x,n_N}^{+}\).  Thus
\[
\begin{aligned}
 &\Ee_{\mathbf Q_{N;0,0,L_N}^{+}}[F G^{\leftarrow}]\\
 &\qquad =
 \Ee_{\mathsf M_{N,m_N}}
 \left[
   F\,D_N^{0,0}\bigl(S_N(m_N)\bigr)
   \Ee_{\widetilde{\mathbf Q}_{N;0,S_N(m_N),n_N}^{+}}[G]
 \right],
\end{aligned}
\]
where \(G^{\leftarrow}\) means that \(G\) is applied to the reversed
second block.

The same uniform estimate holds:
\[
 D_N^{0,0}(x)
 \le
 C\min\left\{1,\frac{x+1}{\sqrt{r_N}}\right\},
\]
Again, heights below \(\delta\sigma_N\sqrt{r_N}\) and above
\(M\sigma_N\sqrt{r_N}\) are negligible as \(\delta, M\to\infty\).
For \(x_N/(\sigma_N\sqrt{r_N})\to z>0\),
\[
 D_N^{0,0}(x_N)
 \longrightarrow
 \sqrt{\frac{2}{\pi s}}\,
 \frac{\mathbf F_0(t-s;z)}
      {\mathbf F_{0,0}(t)},
\]
and the one-zero-endpoint result for the dual array gives
\[
 \widetilde{\mathbf Q}_{N;0,x_N,n_N}^{+}
 \Longrightarrow
 \mathbf Q_{0,z}^{+,t-s}.
\]

Hence the same midpoint cutoff applies.  Finally,
\[
 \rho_s(z)
 \sqrt{\frac{2}{\pi s}}\,
 \frac{\mathbf F_0(t-s;z)}
      {\mathbf F_{0,0}(t)}
 =
 \frac{
   \mathbf F_0(s;z)\mathbf F_0(t-s;z)
 }{
   \mathbf F_{0,0}(t)
 },
\]
which is the midpoint decomposition of the normalized Brownian
excursion.  Reversing the second block back and concatenating proves
the \(0\)-to-\(0\) case.

\end{proof}

\medskip

\textsc{Jiaming Xu, the ohio state university, columbus, oh, usa}

Email: \texttt{jxu0800@gmail.com}

\end{document}